\documentclass[11pt]{amsart}
\usepackage{amsmath, amsfonts, amsthm, amssymb} 
\usepackage{enumitem}
\usepackage{multirow}
\usepackage{hyperref}
\usepackage{cite}
\usepackage[all]{xy}

\hypersetup{
  colorlinks   = true,
  urlcolor     = blue,
  linkcolor    = red,
  citecolor    = blue 
}

\theoremstyle{plain}
\newtheorem{theorem}{Theorem}[section]
\newtheorem{lemma}[theorem]{Lemma}
\newtheorem*{lamptheorem}{Theorem \ref*{lampthm}}
\newtheorem{corollary}[theorem]{Corollary}

\newtheorem{proposition}[theorem]{Proposition}

\theoremstyle{definition}

\newtheorem{example}[theorem]{Example}

\theoremstyle{remark}
\newtheorem*{remark}{Remark}

\newcommand{\K}{\ensuremath{K}}
\newcommand{\R}{\ensuremath{\mathbb{R}}}
\newcommand{\Z}{\ensuremath{\mathbb{Z}}}
\newcommand{\Q}{\ensuremath{\mathbb{Q}}}
\newcommand{\N}{\ensuremath{\mathbb{N}}}
\newcommand{\C}{\ensuremath{\mathbb{C}}}
\newcommand{\F}{\ensuremath{\mathbb{F}}}

\newcommand{\bfu}{\ensuremath{\mathbf{U}}}

\newcommand{\bfa}{\ensuremath{\mathbf{A}}}
\newcommand{\bfb}{\ensuremath{\mathbf{B}}}

\newcommand{\bft}{\ensuremath{\mathbf{T}}}

\newcommand{\bfg}{\ensuremath{\mathbf{G}}}
\newcommand{\torus}{\bft}

\newcommand{\mcv}{\ensuremath{\mathcal{V}}}

\providecommand{\cyc}[1]{\left\langle#1\right\rangle}

\providecommand{\set}[1]{\left\lbrace#1\right\rbrace}
\DeclareMathOperator{\Hom}{Hom}

\DeclareMathOperator{\Aut}{Aut}

\DeclareMathOperator{\Comm}{Comm}

\DeclareMathOperator{\Der}{Der}
\DeclareMathOperator{\VDer}{VDer} 
\DeclareMathOperator{\rank}{rank}

\DeclareMathOperator{\GL}{GL}

\newcommand{\cog}{\ensuremath{\Comm_{\Gamma_0}(\Gamma)}}
\newcommand{\zug}{\ensuremath{Z_\bfu(\Gamma)}}

\newcommand\restr[2]{{
  \left.\kern-\nulldelimiterspace 
  #1 
  \vphantom{\big|} 
  \right|_{#2} 
  }}

\newcommand\suchthat{ \mathrel{}\middle| \mathrel{}}

\renewcommand{\bold}[1]{\medskip \noindent {\bf #1 }\nopagebreak}

\begin{document}

\title{Commensurators of solvable $S$-arithmetic groups}
\author{Daniel Studenmund} \date{\today}
\maketitle

\begin{abstract}
  We show that the abstract commensurator of an $S$-arithmetic
  subgroup of a solvable algebraic group over $\Q$ is isomorphic to
  the $\Q$-points of an algebraic group, and compare this with
  examples of nonlinear abstract commensurators of $S$-arithmetic
  groups in positive characteristic. In particular, we include a
  description of the abstract commensurator of the lamplighter group
  $(\Z/2\Z) \wr \Z$.
\end{abstract}

\section{Introduction}

\bold{Overview.} In this paper we show that the abstract commensurator
of an $S$-arithmetic subgroup of a solvable $\Q$-group is isomorphic
to the $\Q$-points of an algebraic group. We then include examples to
show that the analogous result in positive characteristic does not
hold. As part of these examples, we provide a description of the
abstract commensurator of the lamplighter group.

\bold{Background.} A {\em $\Q$-group} $\bfg$ is a linear algebraic
group defined over $\Q$. For $S$ any finite set of prime numbers, let
$\bfg(S)$ denote the set of {\em $S$-integer points} of $G$; that is,
those matrices in $\bfg(\Q)$ whose entries have denominators with
prime divisors belonging to $S$. A subgroup of $\bfg(\Q)$ is {\em
  $S$-arithmetic} if it is commensurable with $\bfg(S)$. When $S =
\emptyset$, an $S$-arithmetic group is called an {\em arithmetic}
group.

The {\em abstract commensurator} of a group $\Gamma$, denoted
$\Comm(\Gamma)$, is the group of equivalence classes isomorphisms
between finite-index subgroups of $\Gamma$, where two isomorphisms are
equivalent if they agree on a finite-index subgroup of $\Gamma$.

The starting point for our work is the following result, immediate
from the fact that $S$-arithmetic subgroups of $\Q$-groups are
preserved by isomorphism of their ambient $\Q$-groups; see
\cite[Thm5.9, pg269]{platrap}. Let $\Aut_\Q(\bfg)$ denote the group of
$\Q$-defined automorphisms of $\bfg$.
\begin{proposition} \label{qautoms} Suppose $\bfg$ is any
  $\Q$-group. For any finite set of primes $S$, there is
  a natural map $\Theta : \Aut_\Q(\bfg) \to \Comm( \bfg(S) )$.
\end{proposition}

In the case that $\bfg$ is a higher-rank, connected, adjoint,
semisimple linear algebraic group that is simple over $\Q$, rigidity
theorems of Margulis \cite{margulis} imply that the map $\Theta$ of
Proposition \ref{qautoms} is an isomorphism. Similarly, if $\bfg$ is
unipotent then $\Theta$ is an isomorphism by Mal'cev rigidity; see
Theorem \ref{unipcomm}. Moreover, in each of these cases the group
$\Aut(\bfg)$ has the structure of a $\Q$-group such that
$\Aut_{\Q}(\bfg) \cong \Aut(\bfg)(\Q)$.

\bold{Main result.} When $\bfg$ is solvable and not unipotent the
group $\bfg(S)$ is not rigid in the above sense. One approach to
remedying this lack of rigidity is taken in \cite{wittesarith}, where
solvable $S$-arithmetic groups are shown to satisfy a form of
archimedean superrigidity. For solvable arithmetic groups, another
study of this failure of rigidity appears in
\cite{grunewaldplatonov}. Extending these methods, we prove the main
theorem of this paper:
\begin{theorem}\label{mainthm}
  Let $\bfg$ be a solvable $\Q$-group and let $S$ be a finite set of
  primes. Then there is a $\Q$-group $\bfa$ such that $\Comm(\bfg(S))
  = \bfa(\Q)$.
\end{theorem}

The group $\bfa$ is constructed explicitly as an iterated semidirect
product of groups. See Section \ref{mainproofsec} for proof and
details.

When $S = \emptyset$ the arithmetic group $\bfg(S) = \bfg(\Z)$ is
virtually polycyclic, and hence virtually a lattice in a connected,
simply-connected solvable Lie group. In \cite{studenmund} it is shown
that the abstract commensurator of a lattice in a connected,
simply-connected solvable Lie group is isomorphic to the $\Q$-points
of a $\Q$-group. Therefore the $S=\emptyset$ case of Theorem
\ref{mainthm} is a consequence of \cite{studenmund}.

When $S \neq \emptyset$ the group $\bfg(S)$ is no longer necessarily
polycyclic, so different methods are necessary. When $\bfu$ is
a unipotent group, for any set of primes $S$ we have
\[\Comm(\bfu(S)) \cong \Aut(\bfu)(\Q).
\]
In particular the abstract commensurator is independent of $S$. For
example, $\Comm(\Z[1/2]) \cong \Comm(\Z[1/3]) \cong \Q^*$. Note that
for each nontrivial unipotent group this provides an infinite family
of pairwise non-abstractly commensurable groups with isomorphic
abstract commensurator.

When $\bfg$ contains a torus, the abstract commensurator of an
$S$-arithmetic subgroup may depend on $S$. For example, let $\torus$
be the Zariski-closure of the cyclic subgroup generated by the matrix
$\left( \begin{smallmatrix} 2&1 \\ 1&1 \end{smallmatrix}\right)$. Note
that $\torus$ is diagonalizable over $\R$ and over $\Q_{11}$ since $5$
has an 11-adic square root, while $\torus$ is not diagonalizable over
either $\Q$ or $\Q_3$. It follows from Theorem \ref{storus} below that
\[
\torus(\emptyset) \doteq \Z, \qquad \torus(\{3\}) \doteq \Z, \quad
\quad \torus(\{11\}) \doteq \Z^2, \quad \text{ and } \quad
\torus(\{3,11\}) \doteq \Z^2,
\]
where we write $G\doteq H$ if $G$ and $H$ contain isomorphic subgroups
of finite index. Then $\Comm( \torus( \{ 11\} ) )$ and $\Comm( \torus(
\{ 3, 11 \} ) )$ are each isomorphic to $\GL_2(\Q)$, but neither is
isomorphic to $\Comm( \torus( \{3\} ) ) \cong \Q^*$. This dependence
on $S$ appears even for groups whose maximal torus acts faithfully on
the unipotent radical; see Theorem \ref{reducedcomm}.

\bold{Explicit description of commensurator.} A key case is when the
action of any maximal torus of $\bfg$ on the unipotent radical of
$\bfg$ is faithful. Such a solvable algebraic group is said to be {\em
  reduced}. When $\bfg$ is reduced, we have the following explicit
statement whether or not $S = \emptyset$:

\begin{theorem} \label{reducedcomm} Let $\bfg$ be a connected and
  reduced solvable $\Q$-group, let $S$ be a finite set of primes, and
  let $\Delta$ be an $S$-arithmetic subgroup of $\bfg$. Suppose
  $\bfg(S)$ is Zariski-dense in $\bfg$. There is a group isomorphism
  \[
  \Comm(\Delta) \cong \Hom(\Q^{N}, Z(\bfg)(\Q)) \rtimes
  \Aut_\Q(\bfg),
  \]
  where $N$ is the maximum rank of any torsion-free free abelian
  subgroup of $\torus(S)$ for any maximal $\Q$-defined torus $\torus
  \leq \bfg$, and the action is by postcomposition.
\end{theorem}

\begin{remark}
  In the case $S = \emptyset$, Theorem \ref{mainthm} follows from
  Theorem \ref{reducedcomm} by the fact that any solvable arithmetic
  group $\Gamma$ is abstractly commensurable with an arithmetic
  subgroup of a {\em reduced} solvable group. See
  \cite[Thm3.4]{grunewaldplatonov} for a proof of this fact. This is
  possible because arithmetic subgroups of tori are abstractly
  commensurable with arithmetic subgroups of abelian unipotent groups;
  both are virtually free abelian. The same method does not work when
  $S$ is nonempty: $S$-arithmetic subgroups of tori are virtually free
  abelian while $S$-arithmetic subgroups of unipotent groups are not.
\end{remark}

\begin{remark}
  Bogopolski \cite{bogopolski} has computed abstract commensurators of
  the solvable Baumslag-Solitar groups to be
  \[\Comm(BS(1,n)) \cong \Q \rtimes \Q^*.\]
  Theorem \ref{reducedcomm} recovers Bogopolski's result in the case
  that $n$ is a prime power, since $BS(1,p^2)$ is isomorphic to the
  group $\bfg(S)$ where $S = \{ p \}$ and $\bfg = \bfb_2 / Z(\bfb_2)$
  for
  \[
  \bfb_2 = \left\{ \begin{pmatrix} x & z \\  0 & y \end{pmatrix}
    \suchthat xy=1 \right\} \subseteq \GL_2(\C).
  \]
  Note that $BS(1,n^k)$ is a finite-index subgroup of $BS(1,n)$,
  hence the two groups have isomorphic abstract commensurators.

  When $n$ is not a prime power, $BS(1,n)$ is no longer commensurable
  with an $S$-arithmetic group. However, $BS(1,n^2)$ embeds as a
  Zariski-dense subgroup of $(\bfb_2 / Z(\bfb_2) )(S)$ where $S$
  consists of the prime factors of $n$. It may be possible to modify
  the proof of Theorem \ref{reducedcomm} to compute $\Comm( BS(1,n) )$
  for any $n$ from this embedding.
\end{remark}

\bold{Positive characteristic and the lamplighter group.} Above we
have only defined $S$-arithmetic subgroups of $\Q$-groups, but
$S$-arithmetic groups may be defined over any global field. However,
Theorem \ref{mainthm} has no obvious analog for $S$-arithmetic groups
over fields of positive characteristic. Section \ref{poscharsection}
includes examples demonstrating this failure.

A well-known example of a solvable $S$-arithmetic group in
characteristic 2 is the lamplighter group $(\Z/2\Z) \wr \Z$. Section
\ref{appendix} describes the abstract commensurator of the lamplighter
group, with the following main result.
\begin{theorem} 
  \label{lampthm}
  Using the definitions of Equations \ref{vderdefn} and
  \ref{dsinftydefn} of Section \ref{appendix}, there is an isomorphism
  \begin{equation*} \Comm( (\Z/2\Z) \wr \Z ) \cong (
    \VDer(\Z, \K) \rtimes \Comm_\infty(\K) ) \rtimes (\Z/2\Z).
\end{equation*}
\end{theorem}

Using this decomposition we show, for example, that the abstract
commensurator of the lamplighter group contains every finite group as
a subgroup.

\bold{Acknowledgments:} I am grateful to Dave Morris, Bena Tshishiku,
Kevin Wortman, and Alex Wright for helpful discussions. Thanks to
Benson Farb for encouraging me to complete this project, as well as
providing helpful feedback on earlier versions of this paper. I am
extremely grateful to Dave Morris for his help, including detailed
comments on a draft of this paper and pointing out a missing step in
the proof of the main theorem.

\section{Background and definitions} \label{backgroundsec}

For any group $\Gamma$, a {\em partial automorphism} of $\Gamma$ is an
isomorphism between finite-index subgroups of $\Gamma$. Two partial
automorphisms $\phi_1$ and $\phi_2$ are {\em equivalent} if there is
some finite indes $\Delta \leq \Gamma$ so that $\restr{\phi_1}{\Delta}
= \restr{\phi_2}{\Delta}$; an equivalence class of partial
automorphisms is a {\em commensuration} of $\Gamma$. The {\em abstract
  commensurator} $\Comm(\Gamma)$ is the group of commensurations of
$\Gamma$. If $\Gamma_1$ and $\Gamma_2$ are abstractly commensurable
groups then $\Comm(\Gamma_1) \cong \Comm(\Gamma_2)$. We will
implicitly use this fact often in the following discussion.

A subgroup $\Delta \leq \Gamma$ is {\em commensuristic} if
$\phi(\Delta\cap \Gamma_1)$ is commensurable with $\Delta$ for every
partial automorphism $\phi: \Gamma_1 \to \Gamma_2$ of $\Gamma$. Say
that $\Delta$ is {\em strongly commensuristic} if $\phi( \Delta \cap
\Gamma_1) = \Delta \cap \Gamma_2$ for every such $\phi$. If $\Delta$
is commensuristic, restriction induces a map $\Comm(\Gamma) \to
\Comm(\Delta)$. If $\Delta$ is strongly commensuristic, then there is
a natural map $\Comm(\Gamma) \to \Comm(\Gamma / \Delta)$.

A group $\Gamma$ {\em virtually} has a property $P$ if there is a
subgroup $\Delta \leq \Gamma$ of finite index with property $P$. For
any $\Lambda$, a {\em virtual homomorphism} $\Gamma \to \Lambda$ is a
homomorphism from a finite-index subgroup of $\Gamma$ to
$\Lambda$. Two such virtual homomorphisms are {\em equivalent} if they
agree on a finite-index subgroup of $\Gamma$.

By a {\em $\Q$-defined linear algebraic group}, or {\em $\Q$-group},
we mean a subgroup $\bfg \leq \GL_n(\C)$ for some $n$ that is closed
in the Zariski topology and whose defining polynomials may be chosen
to have coefficients in $\Q$. The {\em $\Q$-points} of $\bfg$ are
$\bfg(\Q) = \bfg \cap \GL_n(\Q)$. If $S$ is a finite set of prime
numbers, we define the group of {\em $S$-integers points} of $\bfg$,
denoted $\bfg(S)$, to be the subgroup of elements of $\bfg(\Q)$ with
matrix coefficients having denominators divisible only by elements of
$S$. A subgroup of $\bfg(\Q)$ is {\em $S$-arithmetic} if it is
commensurable with $\bfg(S)$. An abstract group $\Gamma$ is {\em
  $S$-arithmetic} if it is abstractly commensurable with an
$S$-arithmetic subgroup of some $\Q$-group $\bfg$.

Now let $\bfg$ be a solvable $\Q$-group, $S$ a finite set of primes,
and $\Gamma = \bfg(S)$. Since $[\bfg : \bfg^0]<\infty$, we will assume
$\bfg$ is connected. The subgroup $\bfu \leq \bfg$ consisting of all
unipotent elements of $\bfg$ is connected, defined over $\Q$, and is
called the {\em unipotent radical}. For any maximal $\Q$-defined torus
$\torus \leq \bfg$, there is a semidirect product decomposition $\bfg
= \bfu \rtimes \torus$.

For any $\Q$-defined torus $\torus$ and any field extension $F$ of
$\Q$, the {\em $F$-rank} of $\torus$, denoted $\rank_F(\torus)$, is
the dimension of any maximal subtorus of $\torus$ diagonalizable over
$F$. We will use the following special case of \cite[Thm5.12,
pg276]{platrap}.
\begin{theorem} \label{storus}
  Let $\torus$ be a torus defined over $\Q$ and $S$ a finite set of
  prime numbers. Then $\torus(S)$ is isomorphic to the product of a
  finite group and a free abelian group of rank 
  \[N = \rank_\R(\torus) - \rank_\Q(\torus) + \sum_{p\in S}
  \rank_{\Q_p}(\torus). \]
\end{theorem}

If $\bfu$ is a connected unipotent $\Q$-group, then $\Aut(\bfu)$ may
be identified with the automorphism group of the Lie algebra of $\bfu$
and thus has the structure of a $\Q$-group. This structure is such
that $\Aut(\bfu)(\Q) = \Aut_\Q(\bfu)$, where $\Aut_\Q(\bfu)$ is the
group of $\Q$-defined automorphisms of $\bfu$. A solvable $\Q$-group
$\bfg$ is said to be {\em reduced}, or to have {\em strong unipotent
  radical}, if the action of any maximal $\Q$-defined torus on the
unipotent radical is faithful. If $\bfg$ is reduced then $\Aut(\bfg)$
naturally has the structure of a $\Q$-group such that $\Aut(\bfg)(\Q)
= \Aut_\Q(\bfg)$; see \cite[Section 4]{grunewaldplatonov} or
\cite[Section 3]{bauesgrunewald}.

\section{Proof of main theorems} \label{mainproofsec}

In this section we begin the work necessary to prove Theorem
\ref{mainthm}, by way of Theorem \ref{reducedcomm}. Let $\bfg$ be a
connected solvable $\Q$-group, let $S$ be a finite set of prime
numbers, and let $\Gamma \leq \bfg(\Q)$ be an $S$-arithmetic subgroup.
Replacing $\bfg$ by the Zariski-closure of $\Gamma$, we will assume
going forward that $\Gamma$ is Zariski-dense in $\bfg$.

Write $\bfg = \bfu \rtimes \torus$ as above. We will assume without
loss of generality that $\Gamma$ decomposes as $\Gamma = \bfu(S)
\rtimes \Gamma_\torus$ for some finitely generated, torsion-free, free
abelian $S$-arithmetic subgroup $\Gamma_\torus \leq \torus(S)$; see
\cite[Lem5.9]{platrap} and Theorem \ref{storus}.

A group $\Gamma$ is {\em uniquely $p$-radicable} if for
every $\gamma \in \Gamma$ there is a unique element $\delta\in \Gamma$
such that $\delta^p = \gamma$.
\begin{lemma}
  Suppose $\Delta$ is any finite-index subgroup of $\Gamma$ and $p\in
  S$. Then $\Delta \cap \bfu(S)$ is the unique maximal uniquely
  $p$-radicable subgroup of $\Delta$.
\end{lemma}
\begin{proof}
  Since $\Gamma_\torus$ is isomorphic to $\Z^N$ for some $N$, it
  suffices to show that $\bfu(S) \cap \Delta$ is uniquely
  $p$-radicable. Moreover, because the property of being uniquely
  $p$-radicable is inherited by subgroups of finite index, it suffices
  to check that $\bfu(S)$ is uniquely $p$-radicable. It is a standard
  fact that $\bfu$ is $\Q$-isomorphic to a subgroup of the group of
  $n\times n$ matrices with $1$'s on the diagonal, which we denote
  $\bfu_n$. Therefore $\bfu(S)$ is commensurable with a subgroup of
  $\bfu_n(S)$. The desired property is preserved by commensurability
  of torsion-free groups, so it suffices to show that $\bfu_n(S)$ is
  uniquely $p$-radicable. This may easily be done by induction on $n$.
\end{proof}
\begin{corollary} \label{uiscomm} If $S\neq \emptyset$, then $\bfu(S)$
  is strongly commensuristic in $\Gamma$.
\end{corollary}
\begin{remark}
  If $S=\emptyset$ then Corollary \ref{uiscomm} is still true when
  $\bfg$ is reduced. This follows from the fact that $\Gamma \cap
  \bfu$ is the Fitting subgroup of $\Gamma$ for any arithmetic
  subgroup $\Gamma \leq \bfg(\Q)$; see \cite[2.6]{grunewaldplatonov}
  for proof.
\end{remark}
\begin{theorem} \label{unipcomm}
  There is an isomorphism $\Comm( \bfu(S) ) \cong \Aut(\bfu)(\Q)$.
\end{theorem}
\begin{proof}
  Since $\bfu(S)$ has the property that for each $u\in \bfu(\Q)$ there
  is some number $k$ so that $u^k \in \bfu(\Z)$, any partial
  automorphism $\phi$ of $\bfu(S)$ is determined by its values on
  $\bfu(\Z)$. The resulting map $\restr{\phi}{\bfu(\Z)} : \bfu(\Z) \to
  \bfu(\Q)$ uniquely extends to a $\Q$-defined homomorphism $\tilde
  \phi : \bfu \to \bfu$ by a theorem of Mal'cev (see for example the
  proof of \cite[2.11, pg33]{raghunathan}.) Since the dimension of the
  Zariski-closure of $\phi( \bfu(\Z) )$ is equal to the dimension of
  $\bfu$ by \cite[2.10, pg32]{raghunathan}, the map $\hat \phi$ is a
  automorphism of $\bfu$.
  
  The assignment $[\phi] \mapsto \tilde \phi$ gives a well-defined a
  map $\xi : \Comm( \bfu(S) ) \to \Aut(\bfu)(\Q)$. We see that $\xi$
  is injective because $\bfu(S)$ is Zariski-dense in $\bfu$, and $\xi$
  is surjective because every $\Q$-defined automorphism of $\bfu$
  induces a commensuration of $\bfu(S)$ by Proposition \ref{qautoms}.
\end{proof}

Now assume that $\bfg$ is reduced. We prove Theorem \ref{reducedcomm}
using methods following those used to prove Theorems A and C of
\cite{grunewaldplatonov}.

\begin{proof}[Proof of Theorem \ref{reducedcomm}:]
  Let $\bfu$ be the unipotent radical of $\bfg$ and fix a maximal
  $\Q$-defined torus $\torus \leq \bfg$. We assume without loss of
  generality that $\Delta = (\Delta \cap \bfu) \rtimes (\Delta \cap
  \torus)$.

  Suppose $\phi : \Delta_1 \to \Delta_2$ is a partial automorphism of
  $\Delta$. By Corollary \ref{uiscomm} and Theorem \ref{unipcomm},
  $\phi$ induces a $\Q$-defined automorphism $\Phi_\bfu \in
  \Aut(\bfu)$. Define $\alpha : \bfg \to \Aut(\bfu)$ to be the map
  induced by conjugation. Note that $\restr{\alpha}{\torus}$ is
  injective since $\bfg$ is reduced.

  It is straightforward to check that for any $\delta\in \Delta_1$ we
  have 
  \[
  \Phi_u \circ \alpha(\delta) \circ \Phi_u^{-1} = \alpha( \phi(
  \delta) ).
  \]
  It follows that conjugation by $\Phi_u$ preserves $\alpha(\bfg)$
  inside $\Aut(\bfu)$. Conjugation by $\Phi_\bfu$ therefore induces an
  isomorphism between $\alpha(\torus)$ and $\alpha(\torus')$ for a
  different maximal $\Q$-defined torus $\torus'\leq \bfg$, and hence
  an isomorphism $\Phi_\torus : \torus \to \torus'$. Thus $\phi$
  determines a self-map of $\bfg$; for each $g\in \bfg$, write $g = u
  t$ for $u\in \bfu$ and $t\in \torus$ and set
  \[
  \Phi_0(g) := \Phi_\bfu(u) \Phi_\torus(t).
  \]

  One can check that $\Phi_0$ is a $\Q$-defined automorphism of $\bfg$
  extending $\phi$. However, the map $\Comm(\Delta) \to \Aut_\Q(\bfg)$
  defined by $[\phi] \mapsto \Phi_0$ is {\em not} necessarily a
  homomorphism of groups. We will show that $\Phi_0$ can be modified
  in a unique way to produce an automorphism $\Phi$ so that
  $\Phi(\delta) \phi(\delta_1)^{-1} \in Z(\bfg)$ for all $\delta \in
  \Delta$. This condition will guarantee the relation $[\phi] \mapsto
  \Phi$ defines a homomorphism.

  It is straightforward to check from our definitions that $\alpha(
  \Phi_0(\delta) \phi(\delta)^{-1} )$ is trivial for all $\delta \in
  \Delta_1$. Therefore $v(\delta) := \Phi_0(\delta) \phi(\delta)^{-1}$
  defines a function $v:\Delta_1 \to Z(\bfu)(\Q)$. One can check that
  \[
  v(\delta_1 \delta_2) = v(\delta_1) \phi(\delta_1) v(\delta_2)
  \phi(\delta_1)^{-1}. 
  \]
  That is, $\phi$ is a {\em derivation} when $Z(\bfu)(\Q)$ is given
  the structure of a left $\Delta_1$-module by $\delta \cdot z =
  \phi(\delta) z \phi(\delta)^{-1}$ for $\delta \in \Delta_1$ and $z\in
  Z(\bfu)(\Q)$. 

  The derivation $v$ is trivial on $\Delta_1 \cap \bfu$, and therefore
  descends to a derivation $\bar v: \Delta_1 \cap \torus \to
  Z(\bfu)(\Q)$. Since $\torus$ is reductive, there is an invariant
  subspace $V \subseteq Z(\bfu)(\Q)$ such that no element of $V$ is
  fixed by the action of $\torus$, i.e.~$C_V(\torus)$ is trivial. Let
  $v^\perp$ be the component of the derivation $\bar v$ in the
  submodule $V$. From a standard cohomological fact (see \cite[Ch3,
  Thm2**, pg44]{segalbook}), $v^\perp$ is an inner derivation. That
  is, there is some $x\in V$ so that $v^\perp (\delta) = \phi(\delta)
  x \phi(\delta)^{-1} x^{-1}$ for all $\delta \in \Delta \cap
  \torus$. It follows that
  \[
  v(\delta) x \phi(\delta) x^{-1} \phi(\delta)^{-1} \in Z(\bfg)(\Q).
  \]
  When $x$ is viewed as an element of $Z(\bfu)(\Q)$, the choice of $x$
  is unique up to $Z(\bfg)(\Q)$.

  Given $\Phi_0$ and $x$ as above, the assignment $\mu(\phi) = x
  \Phi_0 x^{-1}$ determines a well-defined map 
  \[
  \mu : \Comm(\Delta) \to \Aut(\bfg)(\Q).
  \] 
  One can check using an obvious modification of
  \cite[2.9]{grunewaldplatonov} that $\mu$ is a homomorphism. Because
  $\Gamma$ is Zariski-dense in $\bfg$, the map 
  \[\Theta : \Aut_\Q(\bfg) \to \Comm(\bfg(S))\]
  of Proposition \ref{qautoms} is injective. Therefore $\Theta$ is a
  section of $\mu$, so there is an isomorphism
  \[
  \Comm(\Delta) \cong \ker(\mu) \rtimes \Aut(\bfg)(\Q).
  \]
  
  Now suppose that $[\phi] \in \ker(\mu)$. It follows from the above
  that $\phi$ is a virtual homomorphism $\Delta \to Z(\bfg)(\Q)$
  trivial on $\Delta \cap \bfu$. We can view $\phi$ as a virtual
  homomorphism $\Delta \cap \torus \to Z(\bfg)(\Q)$. Since $\Delta
  \cap \torus$ is virtually $\Z^{N}$, the group of equivalence classes
  of such virtual homomorphisms is isomorphic to $\Hom( \Q^{N},
  Z(\bfg)(\Q) )$. We therefore have a well-defined map
  \[
  \xi : \ker(\mu) \to \Hom( \Q^{N}, Z(\bfg)(\Q) ).
  \]
  Clearly $\xi$ is injective. On the other hand, suppose that $[\Delta
  \cap \torus: \Lambda ] < \infty$ and that $f : \Lambda \to
  Z(\bfg)(\Q)$ is a homomorphism. There is a finite-index subgroup
  $\widetilde \Lambda \leq \Lambda$ so that $f(\widetilde \Lambda)
  \leq Z(\bfg)(S)$. The map 
  \[\phi : \bfu(S) \rtimes \widetilde \Lambda
  \to \bfu(S) \rtimes \widetilde \Lambda
  \]
  defined by $\phi(u, \lambda) = (u\cdot f(\lambda), \lambda)$ induces
  a commensuration of $\Delta$ mapping to $f$ under $\xi$, hence $\xi$
  is surjective. This completes the proof of Theorem \ref{reducedcomm}.
\end{proof}

Now consider the case that $\bfg$ is a connected solvable group, not
necessarily reduced. Assume for the rest of this section that $S \neq
\emptyset$. (The case that $S=\emptyset$ is addressed by the remarks
following the statement of Theorem \ref{mainthm}.) Our primary goal is
to reduce to a situation where Theorem \ref{reducedcomm} can be
applied. This reduction will occur over several steps.

Define $\torus_0 \leq \torus$ to be the centralizer of $\bfu$ in
$\torus$, a $\Q$-defined subgroup of $\torus$. There is a $\Q$-defined
subgroup $\torus_1 \leq \torus$ such that $\torus = \torus_0 \torus_1$
and $\torus_0 \cap \torus_1$ is finite. Without loss of generality we
replace $\bfg$ by $\bfg / (\torus_0 \cap \torus_1)$ and henceforth
assume that $\torus_0 \cap \torus_1 = \{1\}$. Note that now $\bfu
\rtimes \torus_1$ is a reduced solvable $\Q$-group. Moreover, without
loss of generality we replace $\Gamma_\torus$ with $\Gamma_0 \times
\Gamma_1$, where $\Gamma_i \cong \Z^{N_i}$ is an $S$-arithmetic
subgroup of $\torus_i$ for each $i=0,1$. See Theorem \ref{storus} for
the formula used to determine $N_i$.

From the semidirect product decomposition $\Gamma = ( \bfu(S) \times
\Gamma_0 ) \rtimes \Gamma_1$, let us denote elements of $\Gamma$ by
triples $(u, \gamma_0, \gamma_1)$, where $u\in \bfu(S)$ and $\gamma_i
\in \Gamma_i$ for $i=0,1$.

Define $\zug = Z(\Gamma) \cap \bfu$. Clearly we have
\[
Z(\Gamma) = \zug \times \Gamma_0.
\]
If $\Delta$ is any finite-index subgroup of $\Gamma$, then $Z(\Delta)
= \Delta \cap Z(\bfg)$ by Zariski-density of $\Delta$. It follows that
$Z(\Gamma)$ is strongly commensuristic in $\Gamma$.

Any virtual homomorphism $\alpha: \Gamma_0 \to \zug$
determines a partial automorphism $\psi_\alpha$ of $\Gamma$ defined on
an appropriate subgroup of $\Gamma$ by 
\[
\psi_\alpha(u, \gamma_0, \gamma_1) := (u + \alpha(\gamma_0),
\gamma_0, \gamma_1).
\]
Let $\mathcal{V}$ denote the subgroup of $\Comm(\Gamma)$ arising in
this way from equivalence classes of virtual homomorphisms $\Gamma_0
\to \zug$. There is an isomorphism
\[
\mcv \cong \Hom\left( \Q^{N_0}, (Z(\bfg)\cap \bfu)(\Q) \right).
\] 

Define
\[
\cog := \set{ \phi \in \Comm(\Gamma) \suchthat
  \phi(\Gamma_0) \subseteq \Gamma_0 }.
\]
Since $Z(\Gamma)$ and $\bfu(S)$ are each strongly commensuristic in
$\Gamma$, we know that $\zug$ is strongly commensuristic in
$\Gamma$. From this we see that $\cog$ normalizes $\mathcal{V}$. Thus
we may form the (semidirect) product
\[\mathcal{V} \cdot \cog \leq \Comm(\Gamma).\]

\begin{proposition} \label{centralreduction}
  $\mcv \cdot \cog = \Comm(\Gamma)$.
\end{proposition}
\begin{proof}
  Suppose $\phi: H \to K$ is a partial automorphism of $\Gamma$. Since
  $\zug$ is strongly commensuristic, $\phi$ induces a commensuration
  $[\nu] \in \Comm(\Gamma_0)$. There is a function $\alpha : H \cap
  \Gamma_0 \to K \cap \zug$ so that
  \[
  \phi( 0, \gamma_0, 0) = (\alpha(\gamma_0), \nu(\gamma_0), 0) \text{
    for all }\gamma_0 \in H\cap \Gamma_0.
  \]
  In fact, it is easy to check that $\alpha$ is a virtual homomorphism
  $\Gamma_0 \to \zug$.

  Define a virtual homomorphism $\Gamma_0 \to \zug$ by
  $\beta = - \alpha\circ \nu^{-1}$. A straightforward computation
  shows that
  \[
  (\psi_\beta \circ \phi) (0, \gamma_0, 0) = (0, \nu(\gamma_0), 0) \text{
    for all }\gamma \in \Gamma_0.
  \]
  This means that $\psi_\beta \circ \phi \in \cog$, which completes
  the proof.
\end{proof}

We now turn to the task of elucidating the structure of $\cog$. There
is a natural map
\[
\xi : \cog \to \Comm(\Gamma / \Gamma_0).
\]
Define $\Comm_\torus(\Gamma)$ to be the kernel of $\xi$. Because $\Gamma /
\Gamma_0$ is naturally identified with the subgroup $\bfu(S) \rtimes
\Gamma_1 \leq \Gamma$, it is easy to see that $\xi$ is
surjective. Therefore there is a short exact sequence
\begin{equation} \label{commsplit}
  1 \to \Comm_\torus(\Gamma) \to \cog \to \Comm(\Gamma / \Gamma_0)
\to 1.
\end{equation}
Because $\Gamma$ decomposes as a direct product $\Gamma = (\bfu(S)
\rtimes \Gamma_1) \times \Gamma_0$, the sequence (\ref{commsplit})
splits and we can identify $\Comm(\Gamma / \Gamma_0) \cong
\Comm(\bfu(S)\rtimes \Gamma_1)$. Then by Theorem \ref{reducedcomm}
there is an isomorphism
\[
\Comm( \Gamma / \Gamma_0 ) \cong \Hom(\Z^{N_1}, Z(\bfu \rtimes
\torus_1)(\Q)) \rtimes \Aut(\bfu \rtimes \torus_1)(\Q).
\]

\begin{lemma} \label{comm0} Let $\Gamma_i \cong \Z^{N_i}$ for $i=0,1$
  be as above. There is an isomorphism
  \[
  \Comm_\torus(\Gamma) \cong \GL_{N_0}(\Q) \ltimes \Hom(\Q^{N_1} , \Q^{N_0} ),
  \]
  where the action is by postcomposition.
\end{lemma}
\begin{proof}
  There is homomorphism $\Psi : \Comm_\torus(\Gamma) \to \GL_{N_0}(\Q)$
  given by restriction to $\Gamma_0$. Because $\Gamma_0$ splits off as
  a direct product factor, $\Psi$ is surjective and the following
  exact sequence splits:
  \[ 1 \to \ker(\Psi) \to \Comm_\torus(\Gamma) \to \GL_{N_0}(\Q) \to 1.\] 

  The kernel of $\Psi$ is given by equivalence classes of virtual
  homomorphisms $\bfu(S) \rtimes \Gamma_1 \to \Gamma_0$. There are no
  virtual homomorphisms $\bfu(S) \to \Gamma_0$ because $\Gamma_0$ is
  free abelian and every finite-index subgroup of $\bfu(S)$ is
  $p$-radicable for any $p\in S$. Therefore the kernel of $\Psi$ may
  be identified with equivalence classes of virtual homomorphisms from
  $\Gamma_1$ to $\Gamma_0$, which form a group isomorphic to $\Hom(
  \Q^{N_1} , \Q^{N_0} )$.
\end{proof}

We now complete the proof of the main theorem of this paper in the
case $S\neq \emptyset$:

\begin{proof}[Proof of Theorem \ref{mainthm}]
  By Proposition \ref{centralreduction}, we have the decomposition
  \begin{equation} \label{centraleqn}
    \Comm(\Gamma) \cong \Hom\left( \Q^{N_0}, (Z(\bfg) \cap \bfu)(\Q)
    \right) \rtimes \cog.
  \end{equation}
  From the split exact sequence (\ref{commsplit}), there is a
  semidirect product decomposition
  \begin{equation} \label{topprod} \cog \cong \Comm_\torus(\Gamma)
    \rtimes \Comm( \Gamma / \Gamma_0 ).
  \end{equation}
  By Lemma \ref{comm0} we know
  \begin{equation} \label{comm0eqn}
  \Comm_\torus(\Gamma) \cong \GL_{N_0}(\Q) \ltimes \Hom(\Q^{N_1} , \Q^{N_0} ).
  \end{equation}
  By Theorem \ref{reducedcomm} we know
  \begin{equation} \label{reducedeqn}
  \Comm( \Gamma / \Gamma_0 ) \cong \Hom(\Q^{N_1}, Z(\bfu \rtimes
  \torus_1)(\Q)) \rtimes \Aut(\bfu \rtimes \torus_1)(\Q).
  \end{equation}
  
  To understand the action of Equation \ref{centraleqn}, note that
  there are maps from $\cog$ to both $\GL_{N_0}(\Q)$ and
  $\Aut(Z(\bfg)\cap \bfu)(\Q)$. The action factors through these maps,
  and $\GL_{N_0}(\Q)$ and $\Aut( Z(\bfg) \cap \bfu) (\Q)$ act by pre-
  and post-composition, respectively.
 
  The action of Equation \ref{topprod} factors through the map
  \[\Comm(\Gamma / \Gamma_0) \to \Aut(\torus_1)(\Q).\] 
  The action of the latter on $\Comm_\torus(\Gamma)$ is simply
  precomposition in the $\Hom( \Q^{N_1} , \Q^{N_0} )$ factor. Note
  that $\Aut(\torus_1)(\Q)$ is finite by rigidity of tori.
  
  If $\bfa$ and $\bfb$ are $\Q$-groups, and $\bfa$ acts on $\bfb$ so
  that the map $\bfa \times \bfb \to \bfb$ is defined over $\Q$, then
  the semidirect product $\bfb \rtimes \bfa$ has the structure of a
  $\Q$-group. Each of the semidirect products of Equations
  \ref{centraleqn}, \ref{topprod}, \ref{comm0eqn}, and
  \ref{reducedeqn} satisfies this condition. It follows that
  $\Comm(\Gamma)$ has the structure of a $\Q$-group.
\end{proof}

\section{Positive characteristic} \label{poscharsection}

Linear algebraic groups can be defined over arbitrary fields. Let $K$
be a global field and $S$ a set of multiplicative valuations of
$K$. The ring of {\em $S$-integral} elements of $K$, denoted $K(S)$,
is the ring of $x\in K$ such that $v(x) \leq 1$ for each
non-Archimedean valuation $v\notin S$. If $\bfg$ is a linear algebraic
group defined over $K$, let $\bfg(K(S))$ denote the group of matrices
in $\bfg$ with entries in $K(S)$. See \cite[Chapter I]{margulis} for
details.

We will be concerned only with specific examples. In what follows we
use the global field $K = \F_q(t)$, the field of rational functions in
one variable over the finite field with $q$ elements. Choose $S =
\{v_t, v_\infty\}$, where the valuations $v_\infty$ and $v_t$ are
defined as follows: Given any $r \in \F_q(t)$, write $r(t) = t^k (
f(t) / g(t) )$, where $f$ and $g$ are polynomials with nontrivial
constant term and $k\in \Z$. Then define
\[ v_t( r ) = q^{-k} \text{ and } v_\infty( r ) = q^{ \deg(f) +
  k - \deg(g) }.
\]
In this case, $K(S)$ is the ring of Laurent polynomials over $\F_q$,
denoted $\F_q[t,t^{-1}]$.

\begin{example} \label{poscharex1}
  Consider the 1-dimensional additive algebraic group
  \[
  \bfg_a = \left\{ \begin{pmatrix} 1 & * \\ 0 & 1 \end{pmatrix}
  \right\} \subseteq \GL_2.
  \]
  Then $\bfg_a(K(S)) \cong K(S)$ is an $S$-arithmetic group. There is
  an isomorphism of abstract groups 
  \[
  K(S) \cong \bigoplus_{k=-\infty}^\infty \F_q.
  \]

  \begin{proposition} \label{commtoobig}
    For any field $F$ and any linear algebraic group $\bfg$ over $F$,
    there is no embedding $\Comm( K(S) ) \to \bfg(F)$.
  \end{proposition}
  \begin{proof}
    It suffices to treat the case that $\bfg = \GL_d$ for some $d$. We
    will show that $\Comm(K(S))$ contains $\GL_n(\F_q)$ for every
    $n$, which implies that $\Comm(K(S))$ contains every finite
    group. This completes the proof, since $\GL_d(F)$ does not contain
    every finite group. (See for example \cite[Thm5]{serrebounds}.)
    
    For each $n\in \N$, embed $\GL_n(\F_q)$ into $\Comm( K(S) )$
    `diagonally' as follows: Let $V = \oplus_{k=-\infty}^\infty \F_q$,
    and for each $\ell \in \Z$ define a subgroup $V_\ell \leq V$ by
    $V_\ell = \oplus_{k= n \ell}^{n(\ell+1) - 1} \F_q$. Given any
    automorphism $\phi \in \GL_n(\F_q)$, define an automorphism $\Phi
    \in \Aut(V)$ piecewise by $\restr{\Phi}{V_\ell} =\phi$. In this way
    every nontrivial element of $\GL_n(\F_q)$ determines a nontrivial
    commensuration of $V \cong K(S)$.

  \end{proof}
  
  In particular, Proposition \ref{commtoobig} implies that Theorem
  \ref{mainthm} does not hold when $\Q$ is replaced by a global field
  of positive characteristic.
\end{example}

\begin{example}[Lamplighter group] \label{poscharex2}
  Consider the algebraic group
  \[
  \bfb_2 = \left\{ \begin{pmatrix} x & z \\ 0 & y \end{pmatrix}
    \suchthat xy=1 \right \} \subseteq \GL_2.
  \]
  Set $q=2$. The $S$-arithmetic group $\bfb_2( \F_2[t,t^{-1}] )$ is
  isomorphic to the (restricted) wreath product $\F_2^2\wr \Z$, which
  is an index 2 subgroup of the {\em lamplighter group} $\F_2 \wr
  \Z$. The lamplighter group is isomorphic to the semidirect product 
  \[\left(\bigoplus_{\Z} \Z / 2\Z \right) \rtimes \Z,\]
  where the $\Z$ acts by permutation of the $\Z/2\Z$ factors through
  the usual left action on the index set.

  The abstract commensurator of $\F_2 \wr \Z$ is fairly complicated,
  and has not been well-studied. See \S\ref{appendix} for a more
  detailed discussion of $\Comm(\F_2 \wr \Z)$. For now we use the fact
  that $\Comm( \F_2 \wr \Z)$ contains the direct limit
  \[
  \varinjlim_{n\in \N} \Aut(\F_2^{n}),
  \]
  where the maps are the diagonal inclusions of $\Aut( \F_2^{n} )$
  into $\Aut( \F_2^{m} )$ whenever $n \mid m$. It follows now as in
  Proposition \ref{commtoobig} that $\Comm( \bfb_2( \F_2[t,t^{-1}] )
  )$ is not a linear group over any field. This shows that Theorem
  \ref{mainthm} does not apply in positive characteristic even in the
  presence of a nontrivial action by a torus.
\end{example}

\section{Commensurations of the lamplighter group} \label{appendix}

Define $\K$ to be the direct product 
\[ \K := \bigoplus_{\Z} \Z / 2\Z. \] The group of integers $\Z$ acts on
itself by left-translation, inducing an action on $K$ by permutation
of indices. The {\em lamplighter group}, which we will denote by
$\Gamma$ throughout this section, is the semidirect product $\Gamma =
\K \rtimes \Z$. The goal of this section is to show that
$\Comm(\Gamma)$ admits the following decomposition.
\begin{lamptheorem}
  Using the definitions of Equations \ref{vderdefn} and
  \ref{dsinftydefn} below, there is an isomorphism
  \begin{equation} \label{lampdesc} \Comm(\Gamma) \cong ( \VDer(\Z,
    \K) \rtimes \Comm_\infty(\K) ) \rtimes (\Z / 2\Z).
\end{equation}
\end{lamptheorem}
See \cite{houghton} for an analogous description of automorphism
groups of unrestricted wreath products.

Let $e_i\in \Gamma$ be the element of the direct sum subgroup which is
nontrivial only the $i^{th}$ index and let $t \in \Gamma$ be a
generator for $\Z$. By definition we have the relation $t^m e_i t^{-m}
= e_{i+m}$. Then $\Gamma$ is generated by the set $\{e_0,t\}$ and has
the presentation
\[ \Gamma = \cyc{e_0, t \suchthat e_0^2=1 \text{ and } [t^k e_0
  t^{-k}, t^\ell e_0 t^{-\ell}] = 1 \text{ for all } k,\ell\in\Z }.
\]

\begin{lemma} \label{appsurj} The quotient map $\Gamma \to \Gamma /
  \K$ induces a surjective homomorphism $\Theta : \Comm(\Gamma) \to
  \Z/2\Z$.
\end{lemma}
\begin{proof}
  The subgroup $\K\leq \Gamma$ is equal to the set of torsion elements
  of $\Gamma$, and is therefore strongly commensuristic. It follows
  that there is a homomorphism $\Theta : \Comm(\Gamma) \to
  \Comm(\Gamma / \K) \cong \Comm(\Z)$. The nontrivial automorphism of
  $\Z$ induces an automorphism, hence a commensuration, of $\Gamma$ by
  $t\mapsto t^{-1}$ and $e_{i} \mapsto e_{-i}$ for each $i\in \Z$. It
  remains to show that the image of $\Theta$ is contained in $\Aut(\Z)
  \leq \Comm(\Z)$.

  Suppose $\phi : \Delta_1 \to \Delta_2$ is a partial automorphism of
  $\Gamma$. In what follows, let $i=1,2$. Let $\K_i = \K \cap
  \Delta_i$. Choose $g_i \in \Delta_i$ so that its equivalence class
  $[g_i]$ generates the image of the quotient map $\Delta_i \to
  \Delta_i / \K_i$. Let $G_i = \cyc{g_i}$. Note that $\Delta_i$ admits
  a product decomposition $\Delta_i = \K_i G_i$.

  Let $m_i$ be the integer such that $g_i = a t^{m_i}$ for some $a\in
  \K_i$. Replacing $g_i$ with its inverse if necessary, assume that
  $m_i > 0$. Each group $G_i$ naturally acts on $\K / \K_i$. Since $\K
  / \K_i$ is finite, after replacing $g_i$ with a power if necessary
  we assume that the action of $G_i$ on $\K / \K_i$ is trivial for
  both $i=1,2$. Our goal is to prove $m_1 = m_2$.

  One can check that $\phi$ induces an isomorphism $[\K_1, G_1] \cong
  [\K_2, G_2]$, where $[\K_i, G_i]$ is the group generated by
  commutators of the form $[a,g] := aga^{-1}g^{-1}$ for $a\in \K_i$
  and $g\in G_i$. (In fact, in this case we know $[\K_i, G_i]$ is
  equal to the {\em set} of elements of the form $[a, g_i]$, which is
  equal to $[a, t^{m_i}]$, for some $a\in \K_i$. This is helpful in
  understanding the proof of the claim below.)  Since $\phi$ induces
  an isomorphism
  \[\K_1 / [ \K_1, G_1] \cong \K_2 / [\K_2, G_2 ],\]
  the desired result is apparent from the following claim.
  
  \bold{Claim:} There are isomorphisms $\K_i / [ \K_i, G_i] \cong (\Z/
  2\Z)^{m_i}$ for $i=1,2$.

  \bold{Proof of Claim:} Let $H_{m_i} \leq \K$ be the subgroup
  generated by the set $\{ e_0, e_1, \dotsc, e_{m_i-1} \}$. Clearly
  $H_{m_i}$ is isomorphic to $(\Z / 2\Z)^{m_i}$.  Let $P_i = \K_i
  \cap H_{m_i}$, and let $Q_i \leq H_{m_i}$ be a complement to $P_i$
  so that $H_{m_i} = P_i \oplus Q_i$. Now consider the subset $S_i
  \subseteq \K_i$ defined by
  \[
  S_i = \left\{ g\in \K \suchthat g = p [q,g_i] \text{ for some } p
    \in P_i \text{ and } q\in Q_i \right\}.
  \]

  The condition that $G_i$ act trivially on $\K / \K_i$ ensures that
  $[a,g_i] \in \K_i$ for any $a\in \K$, and so $S_i \subseteq
  \K_i$. By construction $S_i$ is in bijection with $H_{m_i}$, hence
  has cardinality $2^{m_i}$. Consider the map of sets $\rho_i : S_i
  \to \K_i / [\K_i , G_i]$ sending an element to its equivalence
  class. Since $[\K_i,G_i]$ consists of elements of the form $[a,g_i]$
  for some $a\in \K_i$, it is not hard to see from the construction of
  $S_i$ that $\rho_i$ is injective. We leave as an exercise to check
  that $\rho_i$ is surjective, which completes the proof.
\end{proof}

Let $\Theta$ be the surjection of Lemma \ref{appsurj}. The short exact
sequence
\[
 1 \to \ker(\Theta) \to \Comm(\Gamma) \to \Z / 2\Z \to 1
\]
splits, so that $\Comm(\Gamma) \cong \ker(\Theta) \rtimes (\Z /
2\Z)$. Since $\K$ is strongly commensuristic, there is a natural map
$\Phi : \ker(\Theta) \to \Comm(\K)$. We first describe the kernel of
$\Phi$ then the image of $\Phi$.

If $G$ is a group and $A$ is a $G$-module, then $\tau : G \to A$ is a
{\em derivation} if $\tau( g_1 g_2 ) = \tau(g_1) + g_1 \cdot
\tau(g_2)$ for all $g_1,g_2\in G$. The set of derivations from $G$ to
$A$ forms an abelian group denoted $\Der(G,A)$. A {\em virtual
  derivation} from $G$ to $A$ is a derivation from a finite-index
subgroup of $G$ to $A$. Two virtual derivations are {\em equivalent}
if they agree on a finite-index subgroup of $G$. The set of
equivalence classes of virtual derivations forms a group
\begin{equation}
  \label{vderdefn}
 \VDer(G, A) := \varinjlim_{[G:H]<\infty} \Der(H, A).
\end{equation}

\begin{lemma} \label{appkerphi}
  There is an isomorphism $\ker(\Phi) \cong \VDer( \Z, \K )$.
\end{lemma}
\begin{proof}
  Given any $[\phi]\in \ker(\Phi)$, find $m\in \Z$ so that $\phi(t^m)$
  is defined. Then define a map $\tau : m\Z \to \K$ by $\tau(t^{k}) =
  \phi(t^{k}) t^{-k}$ for any $k\in m\Z$. It is easy to check that
  $\tau$ is a derivation from $m\Z$ to $\K$, and that the assignment
  $[\phi] \mapsto \tau$ gives a homomorphism $\Comm(\Gamma) \to
  \VDer(\Z, \K)$. This assignment is clearly injective. On the other
  hand, if $\tau \in \Der(m\Z,\K)$ then setting $\phi(x t^\ell) = x
  \tau(t^\ell) t^\ell$ for $x\in \K$ defines an automorphism $\phi$ of
  $\Gamma_m \leq \Gamma$.
\end{proof}

Let $\Comm(\K)^{m\Z}$ denote the group of $m\Z$-equivariant
commensurations of $\K$. There are natural inclusions $\Comm(\K)^{m\Z}
\to \Comm(\K)^{n\Z}$ whenever $m\mid n$. Define
\begin{equation}
 \label{dsinftydefn} 
 \Comm_\infty(\K) := \varinjlim_m \Comm(\K)^{m\Z}.
\end{equation}

\begin{lemma} \label{appimphi} There is an isomorphism $\Phi(\ker(\Theta))
  \cong \Comm_\infty(\K)$.
\end{lemma}
\begin{proof}
  Supose $\alpha = \Phi( [\phi] )$ for some partial automorphism
  $\phi$ of $\Gamma$. Find $m\in \Z$ so that $t^m$ is in the domain of
  $\phi$. Define $x_0 = \phi(t^m)t^{-m} \in \K$. Then given any $x\in
  \K$, we have
\[
\phi(t^m x t^{-m}) = x_0 t^m \phi(x) t^{-m} x_0^{-1} = t^m \phi(x)
t^{-m}.
\]
From this we see that any $\alpha \in \Phi(\ker(\Theta))$ is
$m\Z$-equivariant for some $m$.

On the other hand, suppose $\beta : H_1 \to H_2$ is any partial
automorphism of $\K$ that is $m\Z$-equivariant. Define $\Gamma_m = \K
\rtimes \cyc{t^m}$, an index $m$ subgroup of $\Gamma$. The formula
$\phi(xt^\ell) = \alpha(x) t^\ell$ defines an automorphism $\phi \in
\Aut(\Gamma_m)$. Hence $[\phi]$ is a commensuration of $\Gamma$ which
evidently satisfies $\Phi( [\phi] ) = \beta$.
\end{proof}

\begin{proof}[Proof of Theorem \ref{lampthm}:] It is clear from the
  proof of Lemma \ref{appimphi} that the short exact sequence
\[ 1 \to \VDer(\Z, \K) \to \ker(\Theta) \to \Comm_\infty(\K) \to 1\]
splits. Putting together the results of Lemmas \ref{appsurj},
\ref{appkerphi}, and \ref{appimphi}, we have the semidirect product
description of Equation \ref{lampdesc}:
\[
\Comm(\Gamma) = ( \VDer(\Z, \K) \rtimes \Comm_\infty(\K) ) \rtimes
(\Z / 2\Z).
\]
The action of $\Comm_\infty(\K)$ on $\VDer(\Z, \K)$ is the action by
postcomposition. The factor of $\Z/2\Z$ preserves both $\VDer(\Z, \K)$
and $\Comm_\infty(\K)$, and acts on $\VDer(\Z, \K)$ by precomposition.
\end{proof}

It is not clear whether a more explicit description of
$\Comm_\infty(\K)$ exists, but we can describe some subgroups. For
example, the `diagonal embedding' construction of Proposition
\ref{commtoobig} shows that $\Comm_\infty(\K)$ contains the direct
limit
\[ \varinjlim_m \GL_m(\F_2),\]
where $\GL_m(\F_2)$ includes into $\GL_n(\F_2)$ diagonally whenever $m
\mid n$. Hence $\Comm_\infty(\K)$ contains every finite group.

Note that $\VDer(\Z, \K)$ contains every commensuration induced by
conjugation by some $a \in \K$. However, some elements of $\VDer(\Z,
\K)$ do not arise in this way. For example, any virtual derivation
$\tau : m\Z \to \K$ such that $\tau(t^m)$ is nontrivial in an odd
number of coordinates cannot arise from conjugation.

\bibliography{SAcomm.bib}{}

\providecommand{\bysame}{\leavevmode\hbox to3em{\hrulefill}\thinspace}
\providecommand{\MR}{\relax\ifhmode\unskip\space\fi MR }
\providecommand{\MRhref}[2]{%
  \href{http://www.ams.org/mathscinet-getitem?mr=#1}{#2}
}
\providecommand{\href}[2]{#2}
\begin{thebibliography}{Mar91}

\bibitem[BG06]{bauesgrunewald}
Oliver Baues and Fritz Grunewald, \emph{Automorphism groups of
  polycyclic-by-finite groups and arithmetic groups}, Publ. Math. Inst. Hautes
  \'Etudes Sci. (2006), no.~104, 213--268. \MR{2264837 (2008c:20070)}

\bibitem[Bog10]{bogopolski}
O.~Bogopolski, \emph{Abstract commensurators of solvable {B}aumslag-{S}olitar
  groups}, Max-Planck-Institute of Mathematics Preprint Series (2010), no.~110,
  8 pages.

\bibitem[GP99]{grunewaldplatonov}
Fritz Grunewald and Vladimir Platonov, \emph{Solvable arithmetic groups and
  arithmeticity problems}, Internat. J. Math. \textbf{10} (1999), no.~3,
  327--366. \MR{1688145 (2000d:20066)}

\bibitem[Hou62]{houghton}
C.~H. Houghton, \emph{On the automorphism groups of certain wreath products},
  Publ. Math. Debrecen \textbf{9} (1962), 307--313. \MR{0150213 (27 \#215)}

\bibitem[Mar91]{margulis}
G.~A. Margulis, \emph{Discrete subgroups of semisimple {L}ie groups},
  Ergebnisse der Mathematik und ihrer Grenzgebiete (3) [Results in Mathematics
  and Related Areas (3)], vol.~17, Springer-Verlag, Berlin, 1991. \MR{1090825
  (92h:22021)}

\bibitem[PR94]{platrap}
Vladimir Platonov and Andrei Rapinchuk, \emph{Algebraic groups and number
  theory}, Pure and Applied Mathematics, vol. 139, Academic Press Inc., Boston,
  MA, 1994, Translated from the 1991 Russian original by Rachel Rowen.
  \MR{1278263 (95b:11039)}

\bibitem[Rag72]{raghunathan}
M.~S. Raghunathan, \emph{Discrete subgroups of {L}ie groups}, Springer-Verlag,
  New York, 1972, Ergebnisse der Mathematik und ihrer Grenzgebiete, Band 68.
  \MR{0507234 (58 \#22394a)}

\bibitem[Seg83]{segalbook}
Daniel Segal, \emph{Polycyclic groups}, Cambridge Tracts in Mathematics,
  vol.~82, Cambridge University Press, Cambridge, 1983. \MR{713786 (85h:20003)}

\bibitem[Ser07]{serrebounds}
Jean-Pierre Serre, \emph{Bounds for the orders of the finite subgroups of
  {$G(k)$}}, Group representation theory, EPFL Press, Lausanne, 2007,
  pp.~405--450. \MR{2336645 (2008g:20114)}

\bibitem[Stu13]{studenmund}
Daniel Studenmund, \emph{{A}bstract commensurators of lattices in {L}ie
  groups}, 2013.

\bibitem[Wit97]{wittesarith}
Dave Witte, \emph{Archimedean superrigidity of solvable {$S$}-arithmetic
  groups}, J. Algebra \textbf{187} (1997), no.~1, 268--288. \MR{1425571
  (98g:20073)}

\end{thebibliography}
\bibliographystyle{amsalpha}

\end{document}